\documentclass[11pt]{amsart}
\usepackage{amssymb,amsmath,amsthm,newlfont,enumerate}

\theoremstyle{plain}
\newtheorem{theorem}{Theorem}
\newtheorem{thmx}{Theorem}

\newtheorem{lemx}[thmx]{Lemma}
\newtheorem{corollary}[theorem]{Corollary}
 
\newcommand{\CC}{{\mathbb C}}
\newcommand{\DD}{{\mathbb D}}
\newcommand{\RR}{{\mathbb R}}
\newcommand{\TT}{{\mathbb T}}
\newcommand{\cG}{{\mathcal G}}
\newcommand{\cS}{{\mathcal S}}

\begin{document}

\title{Outer functions and uniform integrability}

\author[Mashreghi]{Javad Mashreghi}
\address{D\'epartement de math\'ematiques et de statistique, Universit\'e Laval, 
Qu\'ebec City (Qu\'ebec),  Canada G1V 0A6}
\email{javad.mashreghi@mat.ulaval.ca}
\thanks{JM supported by an NSERC grant}

\author[Ransford]{Thomas Ransford}
\address{D\'epartement de math\'ematiques et de statistique, Universit\'e Laval, 
Qu\'ebec City (Qu\'ebec),  Canada G1V 0A6}
\email{thomas.ransford@mat.ulaval.ca}
\thanks{TR supported by grants from NSERC and the Canada Research Chairs program}

\date{12 Mar 2018}

\begin{abstract}
We show that, if $f$ is an outer function and $a\in[0,1)$, then the set of functions
\[
\{\log |(f\circ\psi)^*|: \psi:\DD\to\DD \text{~holomorphic}, |\psi(0)|\le a\}
\]
is uniformly integrable on the unit circle. As an application, we derive a simple
proof of the fact that, if $f$ is outer and $\phi:\DD\to\DD$ is holomorphic, then $f\circ\phi$ is outer.
\end{abstract}



\maketitle


\section{Introduction}\label{S:intro}

Let $\DD$ be the open unit disk and $\TT$ be the unit circle.
We write $\cS$ for the set of holomorphic functions $\phi:\DD\to\DD$
(essentially the \emph{Schur class}, except that we exclude constant unimodular functions).

A holomorphic function $f:\DD\to\CC$ is called \emph{outer} if it has the form
\begin{equation}\label{E:outer}
f(z)=c\exp\Bigl(\int_\TT \frac{e^{i\theta}+z}{e^{i\theta}-z}\log\rho(e^{i\theta})\,\frac{d\theta}{2\pi}\Bigr)
\quad(z\in\DD),
\end{equation}
where $c$ is a unimodular constant and  $\rho:\TT\to\RR^+$ is a  function such that $\log\rho\in L^1(\TT)$.
Outer functions are a key tool in the theory of Hardy spaces.
Among their many nice properties is the following folklore fact: if $f$ is outer and $\phi\in\cS$,
then $f\circ \phi$ is also outer. This note arose as an attempt to better understand why this fact is true.

We shall study two classes of functions.
The \emph{Nevanlinna class} $N$ consists of those  functions of the form $f=f_1/f_2$,
where $f_1,f_2$ are bounded and holomorphic on $\DD$ and $f_2$ has no zeros.
The \emph{Smirnov class} $N^+$ is the subclass of $N$ consisting of those $f=f_1/f_2$,
where $f_1,f_2$ are bounded and holomorphic on $\DD$ and $f_2$ is outer.

If $f\in N$, then its radial boundary limits 
\[
f^*(e^{i\theta}):=\lim_{r\to1^-}f(re^{i\theta})
\]
exist a.e.\ on $\TT$.
This is a simple consequence of the corresponding result for bounded holomorphic functions,
due to Fatou.
Also, it is clear that, if $f\in N$ and $\phi\in\cS$, then $f\circ\phi\in N$.
The corresponding result for $N^+$ is also true, but rather less obvious.
As we shall see, it is more or less equivalent to the analogous result for outer functions.

The following theorem lists a number of well-known characterizations of $N^+$.
We write $f_r(z):=f(rz)$. Also, we recall that $f$ is called \emph{inner} 
if it is a bounded holomorphic function on $\DD$ satisfying $|f^*|=1$ a.e.\ on $\TT$.

\begin{thmx}\label{T:N+}
Let $f\in N$. The following statements are equivalent:
\begin{enumerate}[\upshape(i)]
\item $f\in N^+$,
\item $f=f_if_o$, where $f_i$ is inner and $f_o$ is outer,
\item $\lim_{r\to 1^-}\int_0^{2\pi}\log^+|f(re^{i\theta})|\,d\theta=\int_0^{2\pi}\log^+|f^*(e^{i\theta})|\,d\theta$,
\item the set $\{\log^+|f_r^*|:0<r<1\}$ is uniformly integrable on $\TT$.
\end{enumerate}
\end{thmx}

For the equivalence of the first three, see for example  \cite[\S2.5]{Du70}.
A proof of the equivalence of (iii) and (iv) can be found in
\cite[Theorem~A.3.7]{EKMR14}.

Our contribution is the following theorem.
Given $a\in[0,1)$, we write $\cS_a:=\{\psi\in\cS:|\psi(0)|\le a\}$.

\begin{theorem}\label{T:new}
Let $f\in N$ and let $a\in[0,1)$. Then $f\in N^+$ if and only if the set
\[
\{\log^+|(f\circ\psi)^*|:\psi\in\cS_a\}
\]
is uniformly integrable on $\TT$.
\end{theorem}

As observed above, if $f\in N$, then $f\circ\psi\in N$ for all $\psi\in\cS$,
and so $(f\circ\psi)^*$ exists a.e.\ on $\TT$.
Thus the statement of the theorem makes sense.
We shall prove this theorem in \S\ref{S:proof}.

Clearly, if $\phi\in\cS$ and $a\in[0,1)$, then
\[
\{\phi\circ\psi: \psi\in\cS_a\}\subset\cS_b,
\]
where $b:=\sup_{|z|\le a}|\phi(z)|\in[0,1)$.  
Theorem~\ref{T:new} therefore immediately implies the following result,
previously obtained by other methods in \cite{YN78} and \cite{CKS00}.

\begin{corollary}\label{C:N+}
If $f\in N^+$ and $\phi\in\cS$, then $f\circ\phi\in N^+$.
\end{corollary}

We now return to the subject of outer functions.
The link with $N^+$ is furnished by the observation that a nowhere-vanishing
holomorphic function $f$ on $\DD$ is outer if and only if both
$f\in N^+$ and $1/f\in N^+$.
Indeed, the `only if' is obvious, and the `if' is an easy consequence of the characterization
(ii) of $N^+$ in Theorem~\ref{T:N+}.

Combining this remark with Theorem~\ref{T:new}, 
we obtain the following theorem, which we believe to be new.

\begin{theorem}
Let $f\in N$ with no zeros and let $a\in[0,1)$. Then $f$ is outer if and only if the set
\[
\{\log|(f\circ\psi)^*|:\psi\in\cS_a\}
\]
is uniformly integrable on $\TT$.
\end{theorem}

From this, we  deduce the result mentioned at the beginning of the section.

\begin{corollary}
If $f$ is outer and $\phi\in\cS$, then $f\circ\phi$ is outer.
\end{corollary}

\section{Proof of Theorem~\ref{T:new}}\label{S:proof}

The main idea of the proof is to exploit a criterion for uniform integrability
due to de la Vall\'ee Poussin. For convenience, we include a quick proof.

Let $(X,\mu)$ be a measure space and let $\cG$ be a family of measurable complex-valued
functions on $X$. We recall that $\cG$ is \emph{uniformly integrable} if
\[
\sup_{g\in\cG}\int_{\{|g|\ge t\}}|g|\,d\mu\to0 \quad(t\to\infty).
\]

\begin{lemx}\label{L:dlVP}
The family $\cG$ is uniformly integrable on $(X,\mu)$ if and only if there exists
a function $\omega:\RR\to\RR^+$ with $\lim_{t\to\infty}\omega(t)/t=\infty$ such that
\begin{equation}\label{E:sup}
\sup_{g\in\cG}\int_X\omega(|g|)\,d\mu <\infty.
\end{equation}
The function $\omega$ may be chosen to be convex and increasing.
\end{lemx}

\begin{proof}
Suppose $\omega$ exists. Given $\epsilon>0$, choose $t$ such that 
$\omega(s)/s\ge 1/\epsilon$ for all $s\ge t$. Then, for each $g\in\cG$, we have
\[
\int_{\{|g|\ge t\}}|g|\,d\mu\le \int_{\{|g|\ge t\}}\epsilon\omega(|g|)\,d\mu\le
\epsilon\int_X\omega(|g|)\,d\mu\le \epsilon M,
\]
where $M$ is the supremum in \eqref{E:sup}.

Conversely, suppose that $\cG$ is uniformly integrable.
Choose a positive increasing sequence $t_n\to\infty$ such  that, for each $n$,
\[
\sup_{g\in\cG}\int_{\{|g|\ge t_n\}}|g|\,d\mu\le 2^{-n}.
\]
Define $\omega(t):=\sum_{n\ge1}(t-t_n)^+$.
Clearly $\lim_{t\to\infty}\omega(t)/t=\infty$ and,
for each $g\in\cG$, we have
\[
\int_X\omega(|g|)\,d\mu=\sum_{n\ge1}\int_X(|g|-t_n)^+\,d\mu
\le\sum_{n\ge1}\int_{\{|g|\ge t_n\}}|g|\,d\mu\le\sum_{n\ge1}2^{-n}\le1.
\]
Finally, we note that $\omega$, as constructed above, is  convex and increasing.
\end{proof}

\begin{proof}[Proof of Theorem~\ref{T:new}]
By considering $\psi$ of the form $\psi(z):=rz~(0<r<1)$,
we see that the `if' part of the theorem follows from the characterization
of $N^+$ given in Theorem~\ref{T:N+}\,(iv). 

We now turn to the `only if' part.
Let $f\in N^+$.
By Theorem~\ref{T:N+}\,(iv),  the family $\{\log^+|f_r^*|:0<r<1\}$ is uniformly integrable on $\TT$.
Therefore, by Lemma~\ref{L:dlVP}, there exists a convex increasing  function $\omega:\RR\to\RR^+$
with $\lim_{t\to\infty}\omega(t)/t=\infty$ such that
\begin{equation}\label{E:hm}
\sup_{0<r<1}\int_\TT \omega\Bigl(\log^+|f(re^{i\theta})|\Bigr)\,\frac{d\theta}{2\pi}<\infty.
\end{equation}
Now  $\omega(\log^+|f|)$ is subharmonic on $\DD$, because
$\omega$ is a convex increasing function and $\log^+|f|$ a subharmonic function on $\DD$
(see \cite[Theorem~3.4.3(ii)]{AG01}). 
By \cite[Theorem~3.6.6]{AG01}, the condition~\eqref{E:hm} implies that $\omega(\log^+|f|)$  has a harmonic majorant
on $\DD$, let us call it $h$.
Thus, if $\psi\in \cS_a$,
then for all $r\in(0,1)$ we have
\[
\int_\TT\omega\Bigl(\log^+|(f\circ\psi)(re^{i\theta})|\Bigr)\,\frac{d\theta}{2\pi}
\le \int_\TT (h\circ\psi)(re^{i\theta})\,\frac{d\theta}{2\pi}
=h(\psi(0))\le M,
\]
where $M:=\sup_{|z|\le a}h(z)$.
Letting $r\to1^-$ and using Fatou's lemma, we deduce that
\[
\int_\TT \omega\Bigl(\log^+|(f\circ\psi)^*(e^{i\theta})|\Bigr)\,\frac{d\theta}{2\pi}\le M.
\]
Thus
\[
\sup_{\psi\in\cS_a}\int_\TT \omega\Bigl(\log^+|(f\circ\psi)^*(e^{i\theta})|\Bigr)\,\frac{d\theta}{2\pi}<\infty,
\]
and the result now follows by applying Lemma~\ref{L:dlVP}  in the other direction.
\end{proof}


\end{document}